\documentclass{amsart}

\usepackage{hyperref}
\usepackage{amssymb}
\usepackage{graphicx}
\usepackage{color}

\newtheorem{defi}{Definition}[section]
\newtheorem{prop}[defi]{Proposition}
\newtheorem{thm}[defi]{Theorem}
\newtheorem{lem}[defi]{Lemma}
\newtheorem{cor}[defi]{Corollary}

\numberwithin{equation}{section}

\newcommand{\N}{\mathbb{N}}
\newcommand{\Z}{\mathbb{Z}}
\newcommand{\R}{\mathbb{R}}
\newcommand{\He}{\mathbb{H}}
\newcommand{\Haus}{\mathcal{H}}
\newcommand{\D}{\mathcal{D}}
\newcommand{\Le}{\mathcal{L}}
\newcommand{\Mass}{\operatorname{\mathbf M}}

\newcommand{\Lip}{\operatorname{Lip}}

\newcommand{\Hol}{\operatorname{H}}

\newcommand{\spt}{\operatorname{spt}}
\newcommand{\diam}{\operatorname{diam}}

\newcommand{\defl}{\mathrel{\mathop:}=}
\newcommand{\defr}{=\mathrel{\mathop:}}
\newcommand{\im}{\operatorname{im}}
\newcommand{\co}{\operatorname{co}}

\newcommand{\B}{{\rm B}} 
\newcommand{\oB}{{\rm U}}  
\newcommand{\degr}[3]{\operatorname{deg}(#1,#2,#3)} 
\newcommand{\wind}[2]{\operatorname{wind}(#2,#1)} 
\newcommand{\ind}[2]{\operatorname{ind}(#2,#1)} 
\newcommand{\id}{{\rm id}}

\newcommand{\overbar}[1]{\mkern 1.5mu\overline{\mkern-1.5mu#1\mkern-1.5mu}\mkern 1.5mu}

\begin{document}

\title[ H\"older surfaces in the Heisenberg group]{Some properties of H\"older surfaces \\in the Heisenberg group}

\author{Enrico Le Donne}
\address{D\'{e}partement de math\'{e}matiques, Universit\'{e} Paris-Sud 11, 91405 Orsay Cedex}
\email{enrico.ledonne@math.ethz.ch}

\author{Roger Z\"{u}st}
\address{D\'{e}partement de math\'{e}matiques, Universit\'{e} de Fribourg, 1700 Fribourg}
\email{roger.zuest@unifr.ch}
\date{May 1, 2012}
\renewcommand{\subjclassname}{%
 \textup{2010} Mathematics Subject Classification}
.

\begin{abstract}
It is a folk conjecture that for $\alpha > 1/2$ there is no 
$\alpha$-H\"older surface in the subRiemannian Heisenberg group. Namely, it is expected that there is no 
embedding from an open subset of $\R^2$ into the  Heisenberg group that is H\"older continuous of order strictly greater than $1/2$. The Heisenberg group here is equipped with its 
Carnot-Carath\'{e}odory distance.
We show that, in the case that such a surface exists, 
 it cannot be of essential bounded variation and it intersects some vertical line in at least a topological Cantor set.
\end{abstract}

\maketitle

\tableofcontents

\section{Introduction}

As phrased by Gromov in \cite[\S 0.5.C]{gromov2},  the H\"older mapping problem between Carnot-Carath\'{e}odory spaces (CC spaces, for short) is the following. Given two CC spaces $V$ and $W$ and a real number $\alpha\in (0,1)$, describe the spaces of $C^\alpha$ maps $f:W\to V.$
In \cite[\S 2.1]{gromov2}, Gromov showed that if $V$ is a contact $3$-dimensional CC manifold and $\alpha>2/3$, then there is no $f:\R^2\to V$ that is a $C^\alpha$ embedding.
Here and in what follows, $\R^2$ is endowed with the Euclidean distance.
Gromov proved such a nonexistence result by showing the nontrivial fact that any topological surface in $V$ has Hausdorff dimension at least $3$.

Giving examples of $C^{1/2}$ embeddings into contact $3$-dimensional CC manifolds is a triviality. Indeed, by the Ball-Box Theorem, any smooth embedding would give an example.
Since the work of Gromov,
it has been an open problem whether there is any $C^\alpha$ embedding $f:\R^2 \hookrightarrow V$ with $\alpha\in(1/2, 2/3]$.

We focus on the example of a standard contact structure, namely the subRiemannian Heisenberg group.
Since all contact $3$-manifolds are locally contactomorphic, there is no loss for generality, being the problem local. 
Hence we consider the contact structure on $\R^3$ with coordinates $p=(p_x,p_y,p_z)$ where the horizontal distribution is given by 
\[ {\rm span}\left\{ {\partial_1} - \frac{p_y}{2}  {\partial_3},
  { \partial_2} + \frac{p_x}{2}   {\partial_3}\right\}. \]
Since we are only interested in H\"older continuity, instead of using a given CC distance, we may use any other distance that is biLipschitz equivalent to it.
Our choice is the following: for $p,p'\in \R^3$,
\[ d(p,p')^4 = ((p'_x - p_x)^2 + (p'_y - p_y)^2)^2 + (p'_z - p_z - \tfrac{1}{2}(p_xp'_y - p_yp'_x))^2. \] 

We denote by $\He$ the metric space $(\R^3,d)$, while $\R^2$ will always be considered with the Euclidean distance. We refer to $\He$ as the (subRiemannian) Heisenberg group.

In our discussion a very special role is played by the horizontal projection, i.e., the map
\[ \pi : \R^3 \to \R^2,\qquad \pi(x,y,z) \defl (x,y). \] 
Notice that $\pi : \He \to \R^2$ is $1$-Lipschitz.

As a  first result,  we show  that, if $\alpha>1/2$, there are no $C^\alpha$ surfaces in $\He$ with the extra property of having essentially bounded variation. For this latter notion we follow \cite{rado} and review it in Definition~\ref{EBV}.

\begin{thm}
\label{heisenbergintro}
Let $U\subset \R^2$ be an open set in the plane. 
Assume there exists $F : U \to \He$ that is a $C^\alpha$ embedding for some $\alpha > \tfrac{1}{2}$.
Then
\begin{itemize}
\item[i) ] the map $\pi \circ F : U \to \R^2$ is not of essentially bounded variation (cf.~Definition \ref{EBV});
\item[ii) ] in particular,   
\[ \int_{\R^2} \#\{ (\pi \circ F)^{-1}  ( q)\} \, dq = \infty. \]
\end{itemize}
\end{thm}

We remark that in the assumption that a map $F : U \to \He$ is a $C^\alpha$ embedding there is  no requirement on H\"older regularity of the inverse map. Namely, the map $F^{-1}:F(U)\to U$ is only assumed to be continuous.
 
Our second result gives some topological properties of such  $C^\alpha$ surfaces (if  $\alpha>1/2$).
Recall that a topological Cantor set is a metrizable space that is compact, totally disconnected, and has no isolated points. In other words, it is a homeomorphic image of the standard Cantor set.

\begin{thm}
\label{cantorintro}
Let $U\subset \R^2$ be an open set in the plane.
Assume there exists $F : U \to \He$ that is a $C^\alpha$ embedding  for some $\alpha > \tfrac{1}{2}$.
Then
\begin{itemize}
\item[i) ]  the projection $\pi(F(U))$ has nonempty interior;
\item[ii) ]   there is a dense set of points $q \in \pi(F(U))$ such that $(\pi \circ F)^{-1}(q)$ contains a topological Cantor set.
\end{itemize}
\end{thm}

Notice that, since $F$ is an embedding, Theorem \ref{cantorintro} is claiming that there exists vertical lines (i.e., sets of the form $\pi^{-1}(q)$) that intersect the surface $F(U)$ in a Cantor set.

We recall now what it means for a map $\varphi : V \to \R^n$ defined on a bounded open set $V \subset \R^n$ to be of {\em essentially bounded variation}.
Since for us it will be the case, we may assume that $\varphi$ has a continuous extension to the compact set $\overbar V$. Consider a point $q \in \R^n$. A set $D$ is an {\em indicator domain} for $(q,\varphi,V)$ if:
\begin{enumerate}
\item $D$ is a connected open subset of $\R^n$,
\item $\overbar D \subset V$,
\item $q \notin \varphi(\partial D)$ and 
\item $\degr{q}{\varphi}{D} \neq 0$.
\end{enumerate}
Here we denote by ${\rm deg}$ the mapping degree, see the next section for some basic facts.
We define a {\em multiplicity function} at $p$ by
\[ K(q,\varphi,V) \defl \sup_{\mathcal S} \sum_{D \in \mathcal S} |\degr{q}{\varphi}{D}|, \]
where the supremum is taken over all collections $\mathcal S$ of pairwise disjoint indicator domains for $(q,\varphi,V)$, see \cite[II.3.2]{rado}.



\begin{defi}[Essentially bounded variation]\label{EBV}
Let $V \subset \R^n$ be  an open bounded set.
A continuous map $\varphi : \overbar V \to \R^n$
is said to be of {\em essentially bounded variation} if
\[ \int K(q,\varphi,V) \, dq < \infty, \]
where $K(\cdot,\varphi,V) $ denote the multiplicity function, which we just defined above.
\end{defi}
Recall that by \cite[II.3.2 Theorem 3]{rado}, the functions  $K(q,\varphi,V)$ is nonnegative and lower semi-continuous in $q$ and therefore also Lebesgue measurable. 
\\



The paper is organized as follows.
In Section 2, we review some notions and some previous results.
A part from setting the terminology, we recall some properties of mapping degree, winding number, and currents.
We remark how, on the plane, a H\"older curve of order strictly greater than $1/2$ induces a well-defined $1$-current.
In Section 3, we prove Theorem \ref{cantorintro}. Initially we recall the observation that in the subRiemannian Heisenberg group a
H\"older curve of order strictly greater than $1/2$ is uniquely determined by its projection. In other words, the projection can be uniquely lifted and such a lift is done via the use of  currents or via the use of winding numbers and areas of components of the complement of the curve, see Lemma \ref{notinjlem}.  
Subsequently, we focus on $\alpha$-H\"older surfaces, with  $\alpha > 1/2$. In Lemma \ref{lemstar} we show the first crucial fact: on each surface there are closed curves that have positive winding number with respect to some vertical line.
From such a lemma, it will be easy to show Theorem \ref{cantorintro}, see Theorem \ref{cantor} for the construction of the Cantor set.
In Section 4, we prove Theorem \ref{heisenbergintro}. Parts i) and ii) of the theorem are discussed in Theorem \ref{4.2} and Corollary \ref{4.3}, respectively.
Actually, Property ii) of Theorem \ref{heisenbergintro} follows from Property i) by a general fact.
Namely, if a map has bounded variation, then it has essentially bounded variation.
We give a self-contained proof of this latter fact, for our specific case, in Section 5.
\\

Both authors would like to thank the ETH Z\"urich, for the excellent working environment, 
when part of this research was conducted.

\section{Preliminaries: Euclidean H\"older curves and induced currents}

Let us first fix some the notation. If $(X,d_X)$  is a metric space and $A \subset X$, then $\B(A,r) \defl \{x \in X : d_X(A,y) \leq r\}$ and $\oB(A,r) \defl \{x \in X : d_X(A,y) < r\}$ denote the closed and the open $r$-neighborhoods of $A$, respectively. If $A$ consists of a single point $p$ these sets are the closed and the open balls of radius $r$ centered at $p$.

Next we recall the definition of H\"older maps.

\begin{defi}[H\"older constant {$\Hol(f)$}]
For $\alpha\in(0,1)$, a map $f: X\to Y$ between metric spaces $(X,d_X)$ and $(Y,d_Y)$ is said to be {\em H\"older of order $\alpha$} (or simply, we say that $f$ is $C^\alpha$) if there exists a constant $K < \infty$ such that, for all $x,x'\in X$,
\[ d_Y(f(x),f(x'))\leq K \left(d_X(x,x')\right)^\alpha. \]
In this case, the infimum over all such $K$ is denoted by $\Hol^\alpha(f)$.
\end{defi}

Let $U \subset \R^n$ be a bounded open set and $\varphi: \overbar U \to \R^n$ a continuous map. For every point $q \in \R^n \setminus \varphi(\partial V)$ the {\em mapping degree} of $\varphi$ at $q$ is an integer denoted by $\degr{q}{\varphi}{U}$. For the exact definition and the following properties we refer to \cite{deg}.

\begin{itemize}
\item (locality property) If $K \subset \overbar U$ is closed and $q \notin \varphi(K \cup \partial U)$, then
\[ \degr{q}{\varphi}{U} = \degr{q}{\varphi}{U\setminus K}. \]
\item (sum property) Let $U$ be a disjoint union of open sets $U_i$. Then
\[ \degr{q}{\varphi}{U} = \sum_i \degr{q}{\varphi}{U_i}, \]
in case that all the degrees are defined.
\item (homotopy invariance) Let $H : [0,1] \times \overbar U \to \R^n$ be a continuous map and let $\gamma : [0,1] \to \R^n$ be a continuous path such that $\gamma(t) \notin H_t(\partial U)$ for $0 \leq t \leq 1$. Then $\degr{\gamma(t)}{H_t}{U}$ does not depend on $t$, see \cite[IV Proposition 2.4]{deg}.
\item (multiplication formula) Let $V,W \subset \R^n$ be bounded open sets. Let $\varphi : \overbar V \to \R^n$ and $\psi : \overbar W \to \R^n$ continuous maps such that $\varphi(\overbar V) \subset W$. The open set $W \setminus \varphi(\partial V)$ decomposes into countably many connected components $W_l$. If $q \in \R^n \setminus \psi(\partial W \cup \varphi(\partial V))$, then
\[ \degr{q}{\psi \circ \varphi}{V} = \sum_l \degr{q}{\psi}{W_l}\degr{W_l}{\varphi}{V}, \]
see \cite[IV Proposition 6.1]{deg}.
\end{itemize}

Let $\varphi, \psi : \overbar U \to \R^n$ be two continuous extensions of a map $\gamma : \partial U \to \R^n$ and $q \notin \gamma(\partial U)$, then
\[ \degr q \varphi U = \degr q \psi U, \]
see \cite[IV Proposition 2.6]{deg}. Such an extension of $\gamma$ always exists by the Tietze Extension Theorem. The {\em winding number} of $q$ with respect to $\gamma$ is denoted by $\wind \gamma q$ and defined as the degree of such an extension. The winding number (respectively the degree) is constant on connected subsets of $\R^n \setminus \im(\gamma)$. This allows to define $\wind{\gamma}{W}$ for every $W \in \co(\gamma)$.
Here and afterwards, we denote by $\co(\gamma)$  the collection of all connected components of the set $\R^n \setminus \im(\gamma)$.

As an illustration of the winding number consider a map $\gamma : S^1 \to \R^2\setminus\{0\}$. Then $\gamma$ induces a homomorphism on homology
$\gamma_* : H_1(S^1) \to H_1(\R^2\setminus\{0\})$. Since $H_1(S^1) \simeq H_1(\R^2\setminus\{0\}) \simeq \Z$ we have $\gamma_*(1) = k$ for some $k \in \Z$. In this case, $\gamma$ is homotopic to $s \mapsto s^k$ and $k = \wind{\gamma}{0}$.
\\

In our approach we will occasionally make use of the language of currents. We refer to \cite{kirch} and \cite{lang} for a systematic introduction to the subject in metric spaces.
Let us review some notation.
For $k \in \N$ and a locally compact metric space $X$, we denote by $\D_k(X)$ the collection of $k$-dimensional currents in $X$ as defined in \cite{lang}. The currents we will consider live in Euclidean spaces, i.e., $X = \R^n$ for some $n \in \N$, and have compact support. For this reason, the standard reference by Federer also serves our purpose and one can replace $\D_k(\R^n)$ by $k$-dimensional flat chains in $\R^n$ as defined in \cite[4.1.12]{fed} without any complication.
By $\Mass$ we denote the mass of a current.
If $f \in L^1(\R^n)$ has compact support, then $[f]\in \D_n(\R^n)$ is the $n$-current that acts on compactly supported differential $n$-forms $\omega \in \Omega^n_c(\R^n)$ as
\[ [f](\omega) = \int_{\R^n} f \omega. \]
This current has finite mass $\Mass([f]) = \|f\|_{L^1}$. If $A \subset \R^n$ is measurable and bounded, then $[A]$ denotes the current induced by the characteristic function of $A$. For $k \geq 1$ the space of $k$-currents is naturally equipped with a boundary operator $\partial : \D_k(\R^n) \to \D_{k-1}(\R^n)$ by the defining equation $\partial T(\omega) = T(d\omega)$ for $\omega \in \Omega^{k-1}_c(\R^n)$. A current $T \in \D_n(\R^n)$ with finite mass can be restricted to a measurable subset $A \subset \R^n$, the resulting current $T\lfloor A \in \D_n(\R^n)$ has also finite mass. If $\gamma : X \to Y$ is a map, then $\gamma_\# : \D_*(X) \to \D_*(Y)$ is the push forward operator on currents. A priori, the push forward is only defined when $\gamma$ is Lipschitz, however, the work of the second author extends this operator to a class of H\"older maps by reducing its domain to normal currents, see \cite{rz} or \cite{zust} for more details.

The following proposition relates currents in $\R^n$ with maps and their degrees. It is a bit more general than what we will need in the process, but we state it here for completeness.

\begin{prop}[Proposition~4.6 of \cite{rz}]
\label{degprop}
Let $U \subset \R^n$ be a bounded open set with finite perimeter. Let $\gamma : \partial U \to \R^n$ be a map that is H\"older continuous of order $\alpha > \tfrac{n-1}{n}$. Then $\gamma_{\#}(\partial[U])$ has a unique filling $T_\gamma$ in $\D_n(\R^n)$ with compact support, $T_\gamma$ has finite mass, and
\[ T_\gamma \lfloor (\R^n \setminus \im(\gamma)) = [\wind{\gamma}{\cdot}] = \sum_{W \in \co(\gamma)} \wind{\gamma}{W}[W], \]
\[ \Mass(T_\gamma \lfloor (\R^n \setminus \im(\gamma))) = \int_{\R^n \setminus \im(\gamma)} |\wind{\gamma}{q}| \, dq = \sum_{W \in \co(\gamma)} |\wind{\gamma}{W}| \Le^n(W). \]
If the ${(n-1)}$-dimensional Hausdorff measure of $\partial U$ is
finite,
then the following two equations hold
\[ T_\gamma = [\wind{\gamma}{\cdot}] = \sum_{W \in \co(\gamma)} \wind{\gamma}{W}[W], \]
\[ \Mass(T_\gamma) = \int_{\R^n} |\wind{\gamma}{q}| \, dq  = \sum_{W \in \co(\gamma)} |\wind{\gamma}{W}| \Le^n(W). \]
\end{prop}



We will only apply the proposition above in dimension $2$ and in the context of curves. Assume now that $\gamma : [0,1] \to \R^2$ is a closed curve that is H\"older continuous of order $\alpha > \tfrac{1}{2}$. By Proposition~\ref{degprop} we obtain a unique filling $T_\gamma \in \D_2(\R^2)$ of $\gamma_\#[[0,1]]$ given by
\[ T_\gamma = \sum_{W \in \co(\gamma)} \wind \gamma W [W]. \]
By abuse of notation we may also write $[\gamma]$ for $\gamma_\#[[0,1]]$.

\begin{lem}
\label{loops}
Let $\gamma : [0,1] \to \R^2$ be a closed H\"older curve of order $\alpha > \tfrac{1}{2}$. Then
\[ \sum_{W \in \co(\gamma)} |\wind{\gamma}{W}| \Le^2(W) 
< \infty \]
and
\begin{equation}
\label{mainformula}
\frac{1}{2}\left(\int_0^1 \gamma_x \, d\gamma_y - \int_0^1 \gamma_y \, d\gamma_x\right) = \sum_{W \in \co(\gamma)} \wind{\gamma}{W} \Le^2(W) .
\end{equation}
\end{lem}

The Riemann-Stieltjes integrals in the statement above exist for this class of H\"older functions by a result of L.C.\ Young, \cite{young}.

\begin{proof}
Let $T_\gamma$ be the filling of Proposition~\ref{degprop}.
The first equation is just stating the fact that $\Mass(T_\gamma)$ is finite. To obtain \eqref{mainformula}, note that
\begin{align*}
 \sum_{W \in \co(\gamma)} \wind{\gamma}{W} \Le^2(W)&=
T_\gamma(dx \wedge dy) \\
& = T_\gamma(1,\pi_x,\pi_y) \\
 & = \frac{1}{2}\left(T_\gamma(1,\pi_x,\pi_y) - T_\gamma(1,\pi_y,\pi_x)\right) \\
 & = \frac{1}{2}\left(\partial T_\gamma(\pi_x,\pi_y) - \partial T_\gamma(\pi_y,\pi_x)\right) \\
 & = \frac{1}{2}\left((\gamma_\#[[0,1]])(\pi_x,\pi_y) - (\gamma_\#[[0,1]])(\pi_y,\pi_x)\right) \\
 & = \frac{1}{2}\left([[0,1]](\pi_x \circ \gamma,\pi_y \circ \gamma) - [[0,1]](\pi_y \circ \gamma,\pi_x \circ \gamma)\right).
\end{align*}
The last term is another expression for the left-hand side of \eqref{mainformula}.
\end{proof}

\section{H\"older surfaces and their intersection with vertical lines}

Initially, we want to clarify to what extent a H\"older curve in the Heisenberg group is the lift of its horizontal projection. The following lemma has been noticed as well by other authors, such as Z.~Balogh, A.~Kozhevnikov, P.~Pansu, J.~Tyson,\ldots.

\begin{lem}
\label{liftlem}
Let $\gamma : [0,T] \to \He$ be a $C^\alpha$ curve  for some $\alpha > \frac{1}{2}$. Then
\[ \gamma_z(t) = \gamma_z(0) + \frac{1}{2} \left( \int_0^t \gamma_x \,d\gamma_y - \int_0^t \gamma_y \,d\gamma_x \right). \]
\end{lem}

\begin{proof}
Let $L = \Hol^\alpha(\gamma)$ and define $z_\gamma$ to be the right-hand side of the equation above. We want to show that $\gamma_z = z_\gamma$. It is obvious that $\gamma_x$ and $\gamma_y$ are $\alpha$-H\"older continuous by the definition of $d$. Let $0 \leq s \leq t \leq T$. We know that
\[ |\gamma_z(t) - \gamma_z(s) - \tfrac{1}{2}(\gamma_x(s)\gamma_y(t) - \gamma_y(s)\gamma_x(t))| \leq d(\gamma(s),\gamma(t))^2 \leq L|t - s|^{2\alpha}. \]
We combine the last inequality  with the following three, see e.g.\ \cite[Corollary~3.4]{zust} or \cite{young},
\begin{align*}
\gamma_x(s)\gamma_y(t) - \gamma_y(s)\gamma_x(t) & = \gamma_x(s)(\gamma_y(t) - \gamma_y(s)) 
   - \gamma_y(s)(\gamma_x(t) - \gamma_x(s)), \\
\left|\int_s^t \gamma_x\, d\gamma_y - \gamma_x(s)(\gamma_y(t) - \gamma_y(s)) \right| & \leq C \Hol^\alpha(\gamma_x)\Hol^\alpha(\gamma_y)|t-s|^{2\alpha} \\
 & \leq CL^2|t-s|^{2\alpha}, \\
\left|\int_s^t \gamma_y\, d\gamma_x - \gamma_y(s)(\gamma_x(t) - \gamma_x(s)) \right| & \leq C \Hol^\alpha(\gamma_x)\Hol^\alpha(\gamma_y)|t-s|^{2\alpha} \\
 & \leq CL^2|t-s|^{2\alpha},
\end{align*}
for some constant $C > 0$ depending only on $\alpha$. 
We obtain
\[ |\gamma_z(t) - \gamma_z(s) - (z_\gamma(t) - z_\gamma(s))| \leq D|t-s|^{2\alpha}, \]
where $D = L + CL^2$. For $t \in [0,T]$ and $n \in \N$ (observe that $\gamma_z(0) = z_\gamma(0)$), we have
\begin{align*}
|\gamma_z(t) - z_\gamma(t)| & \leq |\gamma_z(0) - z_\gamma(0)| \\
 & \qquad + \sum_{i=1}^n\left|\gamma_z(t\tfrac{i}{n}) - \gamma_z(t\tfrac{i-1}{n}) - (z_\gamma(t\tfrac{i}{n}) - z_\gamma(t\tfrac{i-1}{n}))\right| \\
 & \leq \sum_{i=1}^n D n^{-2\alpha} = Dn^{1-2\alpha}.
\end{align*}
Taking the limit $n \to \infty$, we get   $\gamma_z(t) = z_\gamma(t)$, for all $t$.
\end{proof}

Lemma~\ref{liftlem} is already the first result where the bound $\alpha > \tfrac{1}{2}$ is sharp. This follows from the fact that for any bounded set $B \subset \He$ there is a constant $C > 0$ such that
\begin{equation*}
\label{holdhalf}
\frac{1}{C}d_{\rm E}(p,p') \leq d(p,p') \leq Cd_{\rm E}(p,p')^{\frac{1}{2}}
\end{equation*}
for all $p,p' \in B$. Here, $d_{\rm E}$ denotes the Euclidean distance on $\R^3$.
\\

Hereafter we start the discussion on H\"older surfaces in the Heisenberg group. Let $U$ be an open set in the Euclidean plane and let $F : U \to \He$  be a H\"older embedding of order $\alpha > \tfrac{1}{2}$. We set
\[ F_h \defl \pi \circ F, \]
and call $F_h$ the {\em horizontal part} of $F$. Recall that, since $\pi$ is $1$-Lipschitz, we have $\Hol^\alpha(F_h) \leq \Hol^\alpha(F)$ for all $\alpha$.

\begin{lem}
\label{notinjlem}
Let $\gamma : S^1 \to U$ (resp.\ $\gamma : [0,1] \to U$) be a closed Lipschitz curve and $\tilde \gamma \defl F_h \circ \gamma$. Then
\[ \sum_{W \in \co(\tilde \gamma)} \wind{\tilde \gamma}{W} \Le^2(W) = 0 \]
and $\tilde \gamma$ is not injective.
\end{lem}

\begin{proof}
Lemma~\ref{liftlem} in combination with Lemma~\ref{loops} implies that
\begin{align*}
0 = F_z(\gamma(1)) - F_z(\gamma(0)) & = \frac{1}{2}\left(\int_0^1 \tilde \gamma_x \, d\tilde \gamma_y - \int_0^1 \tilde \gamma_y \, d\tilde \gamma_x \right) \\
 & = \sum_{W \in \co(\tilde \gamma)} \wind{\tilde \gamma}{W} \Le^2(W).
\end{align*}
Assume by contradiction that $\tilde \gamma$ is injective. By the Jordan-Sch\"onflies Theorem there is a homeomorphism $\varphi$ of $\R^2$ such that $\varphi|_{S^1} = \tilde \gamma$. The multiplication formula of degree theory leads to
\[ 1 = \degr{\varphi^{-1}(p)} {\id} {\oB(0,1)} = \degr{\varphi^{-1}(p)}{\varphi^{-1}}{V} \degr{p}{\varphi}{\oB(0,1)} \]
for a point $p$ in the bounded component $V \in \co(\tilde \gamma)$ and hence
\[ |\wind {\tilde \gamma} p| = |\degr{p}{\varphi}{\oB(0,1)}| = 1. \]
Because the unbounded component of $\co(\tilde \gamma)$ has vanishing winding number with respect to $\tilde \gamma$
\[ \sum_{W \in \co(\tilde \gamma)} \wind {\tilde \gamma} W \Le^2(W) = \wind {\tilde \gamma} V \Le^2(V) \neq 0 \]
contradicting the first part of the lemma. 
\end{proof}

As an immediate consequence we get that $\pi$ cannot be injective on $F(U)$. In other words, the surface  $F(U)$ is not a graph with respect to the vertical direction. In Theorem~\ref{cantor} we will see a much stronger statement.


We show now that our H\"older surfaces have a special property of twisting. The following proof is based on the fact that a game of Hex always has a winner.
\begin{lem}
\label{lemstar}
Let $F : U \to \He$  be a H\"older embedding of order $\alpha > \tfrac{1}{2}$ defined on an open set $U\subset\R^2$.
Then
\begin{itemize}
\item[(*) ] for every open set $V \subset U$ there is a Lipschitz curve $\gamma : S^1 \to V$ such that 
the current ${F_h}_\#[\gamma] $ is not $ 0$.
\end{itemize}
In particular, for such a $\gamma$, there exist
a component $W \in \co(F_h \circ \gamma)$ with nonzero winding number $\wind{F_h \circ \gamma}W$.
 \end{lem}

\begin{proof}
Assume by contradiction that we have an open ball $V \subset U$ for which all closed Lipschitz curves contained in it go to zero by applying ${F_h}_\#$. Fix two points $p$ and $q$ in $V$ with $F_h(p) \neq F_h(q)$ (this is possible because otherwise $F(V)$ would be contained in a vertical axis of $\He$). By some rotation and scaling of $V$ we can assume that $p=(-1,0)$ and $q=(1,0)$. Our assumption implies that there is a $1$-current $T \in \D_1(\R^2)$ such that ${F_h}_\#[\gamma] = T$ for every Lipschitz curve $\gamma$ in $V$ connecting $p$ with $q$ (otherwise we could build a loop not going to the zero-current). The current $T$ is not zero because its boundary is $[F_h(q)] - [F_h(p)]$, which is not zero. Let $x \in \spt(T) \setminus \{F_h(p),F_h(q)\}$ and $\epsilon_0 > 0$ such that $\B(x,\epsilon_0)$ does not contain $F_h(p)$ and $F_h(q)$ (this is possible because a non-zero metric $1$-current cannot be supported on finitely many points, see e.g.\ \cite{lang}). By continuity, there is a $\delta > 0$ such that
\[ F_h(\{(-1,t),(1,t) : t\in[-\delta,\delta]\}) \cap \B(x,\epsilon_0) = \emptyset. \]
Again by some scaling we can assume that $\delta = 1$ and that the whole square $[-1,1]^2$ is contained in $V$.

Let $\epsilon \in(0, \epsilon_0]$ and $n \in \N$ satisfying $\Hol^\alpha(F) 2^\frac{\alpha}{2} n^{-\alpha} \leq \epsilon$.
We want to play a Hex game on $Q = (n^{-1} \Z^2) \cap [-1,1]^2$. Two points $a,b \in Q$ are connected if they have the same color and are adjacent in the sense that
\[ \max\{|b_1-a_1|,|b_2-a_2|\} = n^{-1} \text{ and } a_1 + a_2 \neq b_1 + b_2. \]
If every point of $Q$ is colored with either withe or black, then there exists a white path connecting the two vertical faces of $[-1,1]^2$ or a black path connecting the two horizontal faces, this is implied by the Brouwer Fixed Point Theorem, see e.g.\ \cite{hex}. The black points are those contained in $F_h^{-1}(\B(x,\epsilon))$, all others are white. See Figure~\ref{fig3} for an illustration of the situation at hand. 

\begin{figure}[h]
\centering
\includegraphics[width=0.45\textwidth]{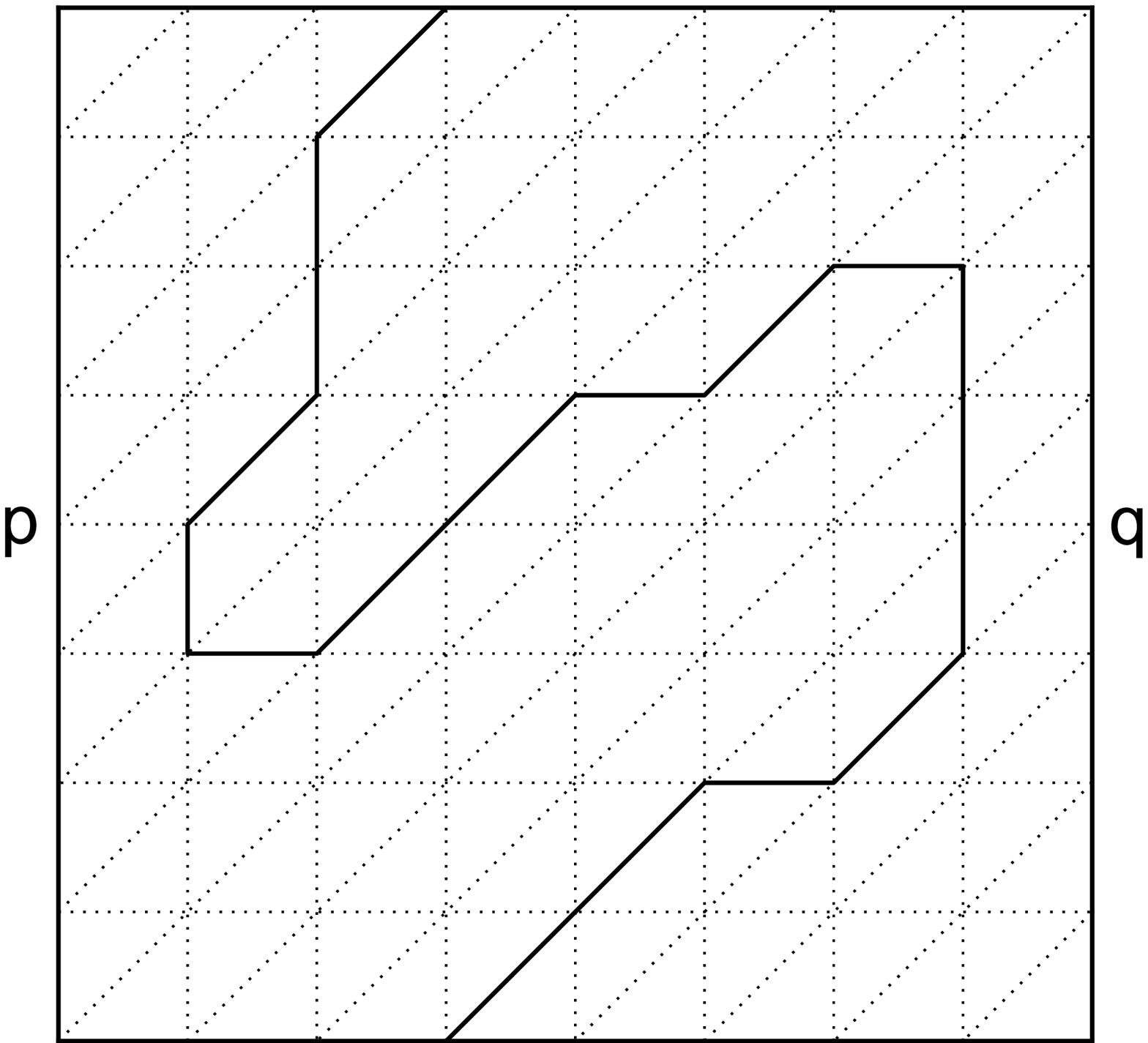}
\caption{The square $[-1,1]^2$ with lattice to the value $n = 4$ and a feasible black path connecting the two horizontal faces. In dotted lines are all the path segments allowed by this lattice.}
\label{fig3}
\end{figure}

We first show that there is no white path from the left to the right face. Assume otherwise. Because both the left and the right faces contain only white points ($\epsilon \leq \epsilon_0$),
 we get a piecewise linear white path $\gamma$ connecting $p$ with $q$. Let $r \in \im(\gamma)$ and $r_Q$ be a white vertex of $\gamma$ with $d(r_Q,r) \leq \sqrt 2 n^{-1}$. Then
\[ d(F_h(r_Q),F_h(r)) \leq \Hol^\alpha(F) d(r_Q,r)^\alpha \leq \Hol^\alpha(F) 2^\frac{\alpha}{2} n^{-\alpha} \leq \epsilon \]
and hence $F_h(r) \neq x$ since $d(F_h(r_Q),x) > \epsilon$. This implies that $x \notin \im(\gamma)$ and also $x \notin \spt({F_h}_\#[\gamma])$ (this is a subset of $\im(\gamma)$). But ${F_h}_\#[\gamma] = T$, a contradiction. 

So there must be a black path $\gamma'$ from top to bottom. Again if $r$ is on this path, we can find a black vertex $r_Q$ with $d(r_Q,r) \leq \sqrt 2 n^{-1}$ leading to
\[ d(x,F_h(r)) \leq d(x,F_h(r_Q)) + d(F_h(r_Q),F_h(r)) \leq 2\epsilon \]
and therefore $F_h(\im(\gamma)) \subset \B(x,2\epsilon)$. 

Choose a sequence $\epsilon_n > 0$ converging to $0$. Based on the preparation above, we can find sequences of points $a_n$ on the top face, $b_n$ on the bottom face and a piecewise linear path $\gamma_n$ inside $[-1,1]^2$ connecting $a_n$ with $b_n$ such that $F_h(\im(\gamma_n)) \subset \B(x,\epsilon_n)$. 
Going to a subsequence if necessary, we can assume that $a_n \to a$ and $b_n \to b$ (also both on the top resp.\ bottom face). By the continuity of $F_h$ we must have $F_h(a) = F_h(b) = x$. Let $c$ be a Lipschitz curve connecting $b$ with $a$ (i.e.\ $\partial[c] = [a]-[b]$).
The curve $F_h \circ c$ is  closed and H\"older  and ${F_h}_\#[c] \neq 0$ by Lemma~\ref{loops} and Lemma~\ref{liftlem} ($F(a)$ and $F(b)$ have to be separated vertically inside $\He$). By the formula in Lemma~\ref{loops}, there is a point $y \in \R^2 \setminus \im(F_h \circ c)$ with $\wind{F_h \circ c}{y} \neq 0$. Let $\epsilon < d(x,y)$ and choose $n$ big enough so that $\epsilon_n \leq \min\{\epsilon_0,\epsilon\}$ and $\Hol^\alpha(F)\max\{d(a,a_n),d(b,b_n)\}^\alpha \leq \epsilon$. 
 Denote by $\gamma_a$ the straight line connecting $a$ with $a_n$ and $\gamma_b$ the straight line connecting $b_n$ with $b$. The concatenation $\gamma \defl c * \gamma_a * \gamma_n * \gamma_b$ is a closed Lipschitz curve such that $y \notin \im(F_h \circ \gamma)$ since $\im(F_h \circ (\gamma_a * \gamma_n * \gamma_b)) \subset \B(x,\epsilon)$. For the same reason, the linear homotopy contracting the closed curve $F_h \circ (\gamma_a * \gamma_n * \gamma_b)$ inside $\B(x,\epsilon)$ to $x$ misses $y$. This shows that $0 \neq \wind{F_h \circ c}{y} = \wind{F_h \circ \gamma}{y}$. But this contradicts ${F_h}_\#[\gamma] = 0$ by Lemma~\ref{loops}.
\end{proof}

In the following lemma we plan to strengthen the property (*) of Lemma \ref{lemstar}. We show that one can  assume that the Lipschitz curve of property (*) is actually the boundary of a triangle.
\begin{lem}
\label{propstar} Let $\alpha>1/2$.
Fix a  $C^\alpha$ embedding $F : U \to \He$. Then for any open $V \subset U$ we can find a simplex $\Delta \subset V$ such that
\[ {F_h}_\# \partial [\Delta] \neq 0. \]
\end{lem}

\begin{proof}
Without loss of generality,  $V$ is some open convex set containing some point $p_0$. By the previous lemma, there is a closed Lipschitz curve $\gamma : S^1 \to V$ such that ${F_h}_\#[\gamma] \neq 0$. The curve $\gamma$ can be approximated by piecewise linear maps $\gamma_n$ such that $\Lip(\gamma_n)$ is bounded. This means that there are points $0 = s_0 < s_1 < \dots < s_{k_n} = 2\pi$ such that, for $s \in [s_i, s_{i+1}]$, we have
\[ \gamma_n(s) = \gamma_n(s_i) + \frac{s - s_i}{s_{i+1} - s_i}(\gamma_n(s_{i+1})-\gamma_n(s_i)). \]
The boundedness of the Lipschitz constants implies that ${F_h}_\#[\gamma_n]$ converges weakly to ${F_h}_\#[\gamma]$ because $\Hol^\alpha(F_h \circ \gamma_n)$ is also bounded in $n$ and $\alpha > \tfrac{1}{2}$. Therefore, we can find some $n$ for which ${F_h}_\#[\gamma_n] \neq 0$. Fix this $n$. Let $\lambda : \B(0,1) \to V$ be the Lipschitz extension of $\gamma_n$ defined by $\lambda(ts) = p_0 + t\gamma_n(s)$ for all $s \in S^1$ ($V$ is convex so this makes sense). We get
\[ [\gamma_n] = \lambda_\# \partial [\B(0,1)] = \lambda_\# \partial \sum_{i=1}^{k_n} [c(s_{i-1},s_i)] = \sum_{i=1}^{k_n} \partial \lambda_\# [c(s_{i-1},s_i)], \]
where $c(s_{i-1},s_i) = \{ ts : t \in [0,1], s \in [s_{i-1},s_i] \}$. By construction, the pushforward $\lambda_\# [c(s_{i-1},s_i)]$ is equal to $[\Delta_i]$ for some simplex $\Delta_i \subset V$ and because of
\[ {F_h}_\#[\gamma_n] = \sum_{i=1}^{k_n} {F_h}_\# \partial [\Delta_i] \]
there is at least one $i$ for which ${F_h}_\# \partial [\Delta_i] \neq 0$.
\end{proof}

\begin{defi}[Set $I_F$ of irregular points]
Let $F : U \to \He$ be a map. A point $p \in U$ is called {\em regular} if it is an isolated point in $F_h^{-1}(F_h(p))$. Otherwise $p$ is called {\em irregular} and we denote by  $I_F \subset U$ the set of {\em irregular points}.
\end{defi}
A regular point $p$ has the property that 
there is an $\epsilon > 0$ such that
\[ F_h^{-1}(F_h(p)) \cap \B(p,\epsilon) = \{p\}. \]
The next result indicates, in terms of the quantity of irregular points,  that the surface $F(U)$ has to be folded quite strongly. 

\begin{thm}
\label{cantor}
For all open $V \subset U$ there is a point $q \in F_h(V)$ for which $F_h^{-1}(q) \cap V$ contains a Cantor set. In particular, $I_F$ is dense in $U$.
\end{thm}


\begin{proof}
By Lemma~\ref{propstar} and Lemma~\ref{lemstar} there is a simplex $\Delta \subset V$ such that ${F_h}_\#[\partial \Delta] \neq 0$. By Lemma~\ref{notinjlem} this means that there are components $W^+_\Delta$ and $W^-_\Delta$ of $\R^2 \setminus F_h(\partial\Delta)$ that have positive resp.\ negative winding number with respect to $\partial \Delta$ ($\partial \Delta$ is given the standard counterclockwise orientation). Now consider the set $\Delta^+ \defl F_h^{-1}(W^+_\Delta) \cap \mathring\Delta$. Repeating the same procedure with $\mathring \Delta^+$ in place of $V$, there is a simplex $\Delta' \subset \mathring \Delta^+$ and a component $W^-_{\Delta'}$ as before. Define the two sets
\begin{align*}
V_0 & \defl F_h^{-1}(W^-_{\Delta'}) \cap \Delta^+ \setminus \Delta', \\
V_1 & \defl F_h^{-1}(W^-_{\Delta'}) \cap \mathring\Delta'.
\end{align*}
By the sum property for the degree
\[ \degr{q}{F_h}{\Delta^+} = \degr{q}{F_h}{V_0} + \degr{q}{F_h}{V_1} \]
for any point $q \in W_{\Delta'}^-$. 
We have that $\degr{q}{F_h}{V_1}$ is negative and $\degr{q}{F_h}{\Delta^+}$ is positive. Hence, $\degr{q}{F_h}{V_0}$ is positive and $q$ has to be in the image $F_h(V_0)$. In particular, $V_0$ and $V_1$ are two open sets with closures contained inside $V$ and $F_h(V_0) = F_h(V_1) = W^-_{\Delta'} \defr W_1$. By taking a smaller simplex for $\Delta$ we can ensure that $\diam(V_0)$ and $\diam(V_0)$ are as small as we want. Similarly by taking a smaller set for $W^-_{\Delta'}$ in the construction of $V_0$ and $V_1$ above, we can ensure that the closures of $V_0$ and $V_1$ are disjoint.

Assume now that we have constructed the open sets $W_1,\dots,W_n$ as well as the open sets $V_\omega$ for all words $\omega$ in letters $0$ and $1$ with length $|\omega| \leq n$. Further, assume these sets satisfy:
\begin{enumerate}
	\item $\overbar V_{\omega} \subset V_{\omega'}$ if $\omega' \mu = \omega$ for some nonempty word $\mu$, i.e., $\omega'$ is a proper beginning of $\omega$,
	\item $\overbar V_\omega \cap \overbar V_{\omega'} = \emptyset$ if of the two words $\omega$ and $\omega'$ none is a proper beginning of the other,
	\item $F_h(V_\omega) = W_{|\omega|}$,
	\item $\diam(V_\omega) \leq 2^{-|\omega|}$.
\end{enumerate}
Note that (1) and (3) together imply that $\overbar W_{i+1} \subset W_i$.

Let $\omega_1, \dots, \omega_{2^n}$ be an enumeration of all words of length $n$. First, the construction above is repeated for the set $V_{\omega_1}$ in place of $V$ to obtain open sets $V_{\omega_1 0}'$, $V_{\omega_1 1}' \subset V_{\omega_1}$ and $W_{n+1}^1$ satisfying the 4 properties above (with an appropriate renaming of the sets).
Next, apply this to $V_{\omega_2} \cap F_h^{-1}(W^1_{n+1})$ to obtain $V_{\omega_2 0}'$,$V_{\omega_2 1}'$ and $W_{n+1}^2 \subset W_{n+1}^1$. We proceed $2^n$ times until the sets $V_{\omega_{2^n}0}$, $V_{\omega_{2^n}1}$ and $W_{n+1}^{2^n}$ are constructed out of $V_{\omega_{2^n}} \cap F_h^{-1}(W^{2^n-1}_{n+1})$. $W_{n+1}^{2^n}$ is contained in all the $W_{n+1}^{i}$'s so we let $W_{n+1} \defl W_{n+1}^{2^n}$ and $V_\omega \defl V_\omega' \cap F_h^{-1}(W_{n+1})$ for all $|\omega| \leq n+1$. The 4 properties hold for these new sets and recursively we can construct the sets $V_\omega$ and $W_{|\omega|}$ for all finite words $\omega$.

The equality
\[ \bigcap_{i\geq1} W_i = \bigcap_{i\geq1} \overbar W_i \]
holds by (1) and (3). This set consists of a single point $q$ by (3) and (4) and the completeness of $\R^2$. Further,
\[ C \defl \bigcap_{i\geq1} \bigcup_{|\omega| \leq i} \overbar V_\omega \]
is a Cantor set and $F_h(C) = \{q\}$. In conclusion, $F_h^{-1}(q) \cap V$ contains a Cantor set.
\end{proof}

Although the set of points $q$ for which $\#\{F_h^{-1}(q)\} = \infty$ is dense in $F_h(U)$, it could be a set of measure zero. In the subsequent section we will look into the sets $\{q : \#\{F_h^{-1}(q)\} \geq k\}$ and their measures in more detail.
We finish this section by joining the previous results to prove Theorem~\ref{cantorintro} of the introduction.

\proof[Proof of Theorem \ref{cantorintro}.]
The first part is implied by Lemma~\ref{propstar}. There is a simplex $\Delta \subset U$ and a component $W$ of $\R^2 \setminus F_h(\partial \Delta)$ such that $\degr{W}{F_h}{\mathring \Delta} \neq 0$ and hence $W \subset F_h(U)$. The second part is just Theorem~\ref{cantor} above.
\qed

\section{Projection of essentially bounded variation}

In this section we want to investigate the possibility of $F_h$ having essentially bounded variation in the sense of Definition \ref{EBV}. 
Let $F : U \to \He$ be $\alpha$-H\"older with $\alpha > 1/2$. We   assume that $U$ is bounded. The next result is an immediate consequence of Proposition~\ref{degprop} and the definition of the multiplicity function $K$.

\begin{lem}
\label{simpclosed}
If $V$ is a connected open set such that $\overbar V \subset U$ and $\Haus^1(\partial V) < \infty$, then the filling of ${F_h}_\# (\partial [V])$, called $T_{\partial V}$, coincides with the current induced by the integrable function $\degr{\cdot}{F_h}{V}$. Moreover, 
\[ \Mass(T_{\partial V}) = \int_{\R^2} |\degr{q}{F_h}{V}| \, dq < \infty, \]
\[ T_{\partial V}(dx \wedge dy) = \int_{\R^2} \degr{q}{F_h}{V} \, dq = 0 \]
and $|\degr{q}{F_h}{V}| \leq K(q,F_h,U)$ for almost all $q \in \R^2$. 
\end{lem}

\begin{proof}
From Proposition~\ref{degprop} we know that $T_{\partial V} = [\degr{\cdot}{F_h}{V}]$ and the first two equations are immediate. 
It remains to prove the relation with $K$. 
The set $V$ is a union of countable many pairwise disjoint connected open sets $V_k$. By the sum property for the degree and the definition of $K$,
\[ |\degr{q}{F_h}{V}| = |\sum_k \degr{q}{F_h}{V_k}| \leq \sum_k |\degr{q}{F_h}{V_k}| \leq K(q,F_h,U) \]
holds for every point $q \notin F_h(\partial V)$ (note that $\partial V_k \subset \partial V$ for all $k$). Further, $\Haus^2(F_h(\partial V)) = 0$ because $F_h$ is $C^\alpha$ for some $\alpha > 1/2$ and $\Haus^1(\partial V) < \infty$.
\end{proof}

We can now present a proof of the first part of Theorem~\ref{heisenbergintro} stated in the introduction.

\begin{thm}\label{4.2}
There is no embedding $F : U \to \He$ of H\"older class $\alpha > \tfrac{1}{2}$ such that $F_h$ is of essentially bounded variation.
\end{thm}

\begin{proof}
From Lemma~\ref{propstar} and Lemma~\ref{lemstar} we know that there exists a simplex $\Delta \subset U$ and a component $W$ of $\R^2 \setminus F_h(\partial \Delta)$ such that
\[ \degr {W}{F_h}{\mathring\Delta} \neq 0. \]
Let $V \defl \mathring\Delta \cap F_h^{-1}(W)$. By the locality property
\[ \degr {W}{F_h}{V} = \degr {W}{F_h}{\mathring\Delta} \neq 0. \]
We can approximate $V$ from the inside by an increasing sequence of open sets $V_n$ such that $\partial V_n$ can be covered by finitely many Lipschitz curves. For example, one can take an exhaustion of $V$ by a union of dyadic squares and $V_n$ is the interior of the union of all squares with diameter bigger than $n^{-1}$. The locality property for the degree implies that $\degr{\cdot}{F_h}{V_n}$ converges pointwise to $\degr{\cdot}{F_h}{V}$ on $W$. Due to Lemma~\ref{simpclosed}, $|\degr{q}{F_h}{V_n}| \leq K(q,F_h,U)$, whereas $\int K(q,F_h,U)\, dq < \infty$ because $F_h$ has essentially bounded variation. The Lebesgue Dominated Convergence Theorem implies that
\[ \int_{\R^2} \degr{q}{F_h}{V_n} \, dq \to \int_{\R^2} \degr{q}{F_h}{V} \, dq. \]
But this leads to a contradiction because $\int \degr{\cdot}{F_h}{V_n} \, dq = 0$ for all $n$ and
\[\int_{\R^2} \degr{q}{F_h}{V} \, dq = \Le^2(W)\degr{W}{F_h}{V} \neq 0\]
by construction.
\end{proof}

A particular instance of $F_h$ having essentially bounded variation is when
\[ \int_{\R^2} \#\{F_h = q\} \, dq < \infty. \]
In this case $F_h$ is said to have {\em bounded variation}. With this implication taking for granted at the moment, the following corollary is immediate.

\begin{cor}\label{4.3}
Let $F : U \to \He$ be an embedding of H\"older class $\alpha > \tfrac{1}{2}$. Then
\[ \sum_{k \geq 1} \Le^2(A_k) = \int_{\R^2} \#\{F_h = q\} \, dq = \infty, \]
where $A_k = \{q \in \R^2 : \#\{F_h = q\} \geq k \}$.
\end{cor}

It is a general fact that a map $\varphi : U \to \R^2$ of bounded variation is also of essentially bounded variation, see \cite[VI.2.2 Theorem 4]{rado}. The reason for this is that, for almost all $q \in \R^2$, the preimage $\varphi^{-1}(q)$ consists of regular points only. The harder part is then to show that, for all but countably many regular points $p \in U$, the degree satisfies
\[ \degr{F_h(p)}{F_h}{\oB(p,r)} \in \{-1,0,1\} \]
for all $r$ small enough, i.e., such that $F_h^{-1}(F_h(p)) \cap \B(p,r) = \{p\}$. Because of the special topological setting of a surface projection, we will recover this property for all regular points in the next section.

\section{On the degree of surface projections}

The results of this section are of purely topological nature and we do not need the particular structure of the Heisenberg group or the fact that the embedding $F$ is H\"older continuous. To emphasize this let $G : U \to \R^3$ be an embedding of an open set $U \subset \R^2$. We want to investigate the value of $\degr q {\pi \circ G} V$, where $\pi$ is the projection of $\R^3$ to the $xy$-plane and $V \subset U$ is some open set. As before we abbreviate $G_h \defl \pi \circ G$. It is understood that all the results that follow apply in particular to the H\"older embeddings $F$ of the previous sections. The main result of this section is the following.

\begin{prop}
\label{topdisc2} 
Let $G : U \to \R^3$ be an embedding of an open set $U \subset \R^2$. 
Assume that $V$ is a bounded connected open set such that $\overbar V \subset U$,  $G_h(\partial V) \subset S^1$, and  $\R^2 \setminus \overbar V$ is connected. 
Then
\[ \degr{0}{G_h}{V} \in \{-1,0,1\}. \]
\end{prop}

Before we turn to the proof, we need some technical preparations. We say that a closed curve $\gamma : S^1 \to \R^n$ is in {\em general position} if $\gamma$ is injective outside a finite subset of $S^1$ and every point of $\R^n$ has at most $2$ preimages. One can show that any closed curve into $\R^2$, or into any other $2$-dimensional manifold for that matter, can be approximated by a curve in general position. This approximation and the one that follow are always assumed with respect to the $C^0$-topology.

Let $s,s' \in S^1$. With $[s,s']$ we denote the closed arc in $S^1$ starting from $s$ and connecting it with $s'$ in clockwise direction. 

\begin{lem}
\label{windlem2}
Let $\gamma : S^1 \to \R^2$ be a curve in general position. If $|\wind{\gamma}{0}|>1$, then there is an arc $a = [s,s']$ in $S^1$ such that $\gamma(s) = \gamma(s')$ and $|\wind{\gamma|_a}{0}| = 1$.
\end{lem}

\begin{proof}
As in the proof of Lemma~\ref{notinjlem}, we note that a simple closed curve $\tilde \gamma :  S^1 \to \R^2$ with $0 \notin \im(\tilde \gamma)$ satisfies
\begin{equation}
\label{windingn2}
|\wind{\tilde\gamma}{0}| \leq 1.
\end{equation}
By our assumption, there are only finitely many (unordered) pairs $\{s_1,s_1'\},$ $\dots,$ $\{s_n,s_n'\}$ in $S^1$ such that $s_i \neq s_i'$ but $\gamma(s_i) = \gamma(s_i')$. Denote the set of these pairs by $P_\gamma$. Each pair of points $\{s_i,s_i'\}$ cuts $S^1$ into two closed arcs. Let $A_\gamma$ be the sub-collection of such arcs $b$ with $\wind{\gamma|_{b}}{0} = 0$. We modify $\gamma$ recursively. Set $\gamma_0 = \gamma$ and $\gamma_{i+1}$ is obtained from $\gamma_{i}$ by choosing some $b \in A_{\gamma_i}$ and set $\gamma_{i+1}|_{b} = \text{const}$ and $\gamma_{i+1}|_{S^1 \setminus b} = \gamma_{i}|_{S^1 \setminus b}$. By a further reparametrization in a neighborhood of $b$, in order that $\gamma_{i+1}$ is not constant on $b$, we can achieve that that $\gamma_{i+1}$ is in general position. This neighborhood can be chosen small enough such that $A_{\gamma_{i+1}}$ is equal $A_{\gamma_{i}}$ minus the pairs $\{s,s'\} \in A_{\gamma_{i}}$ with $\{s,s'\} \cap b \neq \emptyset$. It follows that $|P_{\gamma_{i+1}}| < |P_{\gamma_{i}}|$ as well as $|A_{\gamma_{i+1}}| < |A_{\gamma_{i}}|$ and in $k \leq n$ steps we get that $A_{\gamma_k}$ is empty. Now, 
\[ \wind{\gamma}{0} = \wind{\gamma_k}{0} \]
because in each step we removed loops with zero winding number w.r.t.\ $0$. By \eqref{windingn2} we know that $P_{\gamma_k} \neq \emptyset$. Take a pair $\{s,s'\} \in P_{\gamma_k}$ such that one of the two arcs $[s,s']$ or $[s',s]$ contains no other pair of $P_{\gamma_k}$ (this is possible because $P_{\gamma_k}$ is nonempty but finite). Call this arc $a$. Let $\gamma'$ be the restriction of $\gamma_k$ to $a$. 
Then
$\gamma'$ is a closed Jordan curve and thus $|\wind{\gamma'}{0}| \leq 1$ by \eqref{windingn2}. But $\wind{\gamma'}{0} = 0$ is not possible because $A_{\gamma_k}$ is empty. By the construction of $\gamma_k$, the pair $\{s,s'\}$ is also a double point for $\gamma$, i.e., $\{s,s'\} \in A_\gamma$. Finally, we note that $|\wind{\gamma_k|_a}{0}| = |\wind{\gamma|_a}{0}|$ because $\gamma_k|_a$ is obtained from $\gamma|_a$ by removing some loops with zero winding number defined on sub-arcs of $a$ (and a slight reparametrization which does not matter for the winding number). This proves that $|\wind{\gamma|_a}{0}| = 1$.
\end{proof}

This can readily be generalized to particular curves into $\R^3$.

\begin{cor}
\label{windcor2}
Let $\gamma : S^1 \to \R^3$ be a curve in general position with image contained in the cylinder $S^1 \times \R$. If $|\wind{\pi\circ\gamma}{0}| > 1$, then there is an arc $a = [s,s']$ in $S^1$ such that $\gamma(s) = \gamma(s')$ and $|\wind{\pi\circ\gamma|_a}{0}| = 1$.
\end{cor}

\begin{proof}
By some scaling and translation of $\gamma$ in the $z$-direction we can assume that $\im(\gamma) \subset S^1 \times [0,1]$ since these operations do not change $\pi \circ \gamma$. Let $Z \defl S^1 \times [0,1]$ be this closed cylinder. Consider the following deformation $H : [0,1] \times Z \to \R^3$ of $Z$ in $\R^3$ given by
\[ H_s(x,y,z) := ((sz+1)x,(sz+1)y,z). \]
Obviously, $0$ is not in the image of $\pi \circ H$ and $H_0 = \id$. By the homotopy invariance of the winding number we have
\[ \wind{\pi \circ \gamma}{0} = \wind{\pi \circ H_1 \circ \gamma}{0}. \]
The result follows now by the lemma above by noting that $\pi \circ H_1$ is injective: If $(z+1)(x,y,0) = (z'+1)(x',y',0)$, then $(z+1)^2 = (z'+1)^2$ hence $z = z'$ and consequently also $(x,y) = (x',y')$.
\end{proof}

Now we are ready to prove Proposition~\ref{topdisc2}. In case $\partial V$ can be parametrized by a simple closed curve, the statement is a direct consequence of Corollary~\ref{windcor2}. The general case is reduced to this one by an approximation argument.

\begin{proof}[Proof of Proposition~\ref{topdisc2}]
By restricting $G$ if necessary, we can assume that $U$ is bounded and that $G$ has a continuous extension to $\overbar U$. This has the advantage that $G$ as well as $G^{-1}$ are uniformly continuous. The open set $V$ can be approximated from the inside by connected open sets $V'$ such that $\partial V'$ is parametrized by finitely many simple closed Lipschitz curves. This can be achieved for example by representing $V$ as a union of dyadic squares and considering a fixed point of $V$ and its connected component in the interior of the union of all squares bigger than a certain size. Because $V$ and $\R^2 \setminus \overbar V$ are connected we can even assume that the boundary of $V'$ is parametrized by just one simple closed Lipschitz curve $\gamma : S^1 \to V$ (this could be justified, for example,   
by the Jordan Curve Theorem). 

By the connectedness of $V$ and the fact that $G_h(\partial V) \subset S^1$, there is a compact connected set $K_i \subset V$ ($i$ stands for `inside') such that
\[ V \cap G_h^{-1}(0) \subset K_i. \]
Similarly, because $U$ is bounded, there is a compact set $K_o \subset \R^2 \setminus \overbar V$ ($o$ stands for `outside') such that
\[ (\R^2 \setminus \overbar V) \cap G_h^{-1}(0) \subset K_o. \]
Obviously, 
\[ G_h^{-1}(0) \subset K_i \cup K_o \]
since $G_h(\partial V)$ does not contain $0$. Let $K_o'$ be an unbounded closed connected set such that $K_o \subset K_o' \subset \R^2 \setminus \overbar V$. Take the union of $\B(K_o,\delta)$ for some $\delta > 0$ with a piecewise linear ray connecting it to infinity for example (remember $\R^2 \setminus \overbar V$ is path connected). We further assume that the approximation of $V'$ inside $V$ is good enough such that
\begin{equation}
\label{eqKi}
\B(K_i,\epsilon) \subset V'
\end{equation}
and
\begin{equation}
\label{eqKo}
\B(K_o',\epsilon) \subset \R^2 \setminus \overbar V
\end{equation}
for some $\epsilon > 0$. To this end, note that $K_o'$ is closed and disjoint from the compact set $\overbar V$. We can also assume that the approximation of $V'$ in $V$ is good enough such that $d(G_h(p),S^1) < \tfrac{1}{4}$ for all $p \in \partial V'$. This is possible because $G_h(\partial V) \subset S^1$. Consider a piecewise linear approximation $\sigma$ of $G \circ \gamma$ such that $d(\sigma,G \circ \gamma) < \tfrac{1}{4}$. With these bounds
\begin{align*}
d(\sigma(s), S^1 \times \R) & \leq d(\sigma(s), G(\gamma(s)) + d(G(\gamma(s), S^1 \times \R) \\
 & < \frac{1}{4} + \frac{1}{4} = \frac{1}{2}
\end{align*}
for all $s \in S^1$. This allows us to define the closed curve $\tilde \sigma : S^1 \to \R^3$ by post-composing $\sigma$ with the orthogonal projection onto the cylinder $S^1 \times \R$ in $\R^3$. By choosing $\sigma$ appropriately we can assume that $\tilde \sigma$ is in general position. The following equations hold:
\begin{equation}
\label{windapproxeq}
\degr 0 {G_h} V = \degr 0 {G_h}{V'} = \wind {G_h \circ \gamma} 0 = \wind {\pi \circ \tilde \sigma} 0
\end{equation}
The fist equation holds because $V \cap G_h^{-1}(0) \subset K_i \subset V'$ and the locality property for the degree, the second equation by the definition of the winding number and the third one is induced by a linear homotopy $H : [0,1] \times \B(0,1) \to \R^2$ since, by the estimates
\[ 
d(G \circ \gamma,\tilde \sigma)  \leq d(G \circ \gamma,\sigma) + d(\sigma,\tilde \sigma) 
  < \frac{1}{4} + \frac{1}{2} < 1,
\]
the point $0$ is not contained in the image $H([0,1]\times S^1)$. Assume by contradiction that $|\degr 0 {G_h} V| > 1$. Then \eqref{windapproxeq} implies that $|\wind {\pi \circ \tilde \sigma} 0| > 1$ as well. By Corollary~\ref{windcor2} we can find an arc $a = [s_1,s_2]$ in $S^1$ such that $\tilde \sigma(s_1) = \tilde \sigma(s_2)$ and $|\wind{\pi \circ \tilde \sigma|_a}{0}| = 1$. The curve $G \circ \gamma|_a$ is in general not a closed curve. But $G(\gamma(s_1))$ is close to $G(\gamma(s_2))$, the closeness depends on how good the approximation of $V'$ inside $V$ is and how small $d(G\circ\gamma,\sigma)$ is. Because $G^{-1}$ is uniformly continuous, we can make $\gamma(s_1)$ as close to $\gamma(s_2)$ as we want. We construct now a closed curve out of $G\circ\gamma|_a$ by parameterizing the straight line connecting $\gamma(s_2)$ with $\gamma(s_1)$ on $S^1 \setminus a$. We call this curve $\gamma'$. If the approximation is good enough and the fact that $G$ is uniformly continuous, a linear homotopy forces
\[ \pm 1 = \wind {\pi \circ \tilde \sigma|_a} 0 = \wind {G_h \circ \gamma'} 0. \]
For the rest we assume that $d(\gamma(s_1),\gamma(s_1)) < \epsilon$. Then the line connecting $\gamma(s_2)$ with $\gamma(s_1)$ can't intersect $K_i$ resp.\ $K_o'$ because otherwise $\gamma(s_1)$ and $\gamma(s_2)$ would be contained in $\B(K_i,\epsilon)$ resp.\ $\B(K_o,\epsilon)$ contradicting \eqref{eqKi} resp.\ \eqref{eqKo} because the image of $\gamma$ has distance bigger than $\epsilon$ from $K_i$ resp.\ $K_o'$. The sets $K_i$ and $K_o'$ are connected, hence there are components $W_i$ and $W_o$ in $\co(\gamma')$ such that $K_i \subset W_i$ and $K_o' \subset W_o$. The locality property and the multiplication formula for the degree together with the observation that $W_o$ is the unique unbounded component of $\co(\gamma')$ (since $K_o'$ is unbounded) imply that
\begin{align*}
\pm 1 & = \wind {G_h \circ \gamma'} 0 \\
 & = \wind {\gamma'} {W_i} \degr 0 {G_h} {W_i} + \wind {\gamma'} {W_o} \degr 0 {G_h} {W_o} \\
 & = \wind {\gamma'} {W_i} \degr 0 {G_h} {W_i}.
\end{align*}
Therefore,
\[ 1 = |\degr 0 {G_h} {W_i}| = |\degr 0 {G_h} {V}|, \]
by successively using the equation above and then the locality property for the degree together with the inclusion $V \cap G_h^{-1}(0) \subset K_i \subset W_i$. But this contradicts our assumption $|\degr 0 {G_h} V| > 1$.
\end{proof}

Instead of taking $S^1$ in the projection we can take any other simple closed curve $C \subset \R^2$.

\begin{cor}
Let $G : U \to \R^3$ be an embedding of an open set $U \subset \R^2$, $C \subset \R^2$ a simple closed curve and $q \in \R^2 \setminus C$. Assume that $V$ is a bounded connected open set such that $\overbar V \subset U$, $G_h(\partial V) \subset C$ and $\R^2 \setminus \overbar V$ is connected. Then
\[ \degr{q}{G_h}{V} \in \{-1,0,1\}. \]
\end{cor}

\begin{proof}
If $q$ is not in the bounded component of $\R^2 \setminus C$, it is obvious that $\degr{0}{G_h}{V}$ vanishes. Otherwise there is a homeomorphism $\varphi : \R^2 \to \R^2$ such that $\varphi(C) = S^1$ and $\varphi(q) = 0$ due to the Jordan-Sch\"onflies Theorem. If we consider $G' \defl (\varphi \times \text{id}_{\R}) \circ G$ and apply Proposition~\ref{topdisc2}, we get the result.
\end{proof}

As indicated at the end of the last section, the fact that $G_h$ has bounded variation, i.e.,\ $\int_{\R^2} \#\{G_h = q\} \, dq < \infty$, implies that $G_h$ has essentially bounded variation. With the help of Proposition~\ref{topdisc2} this is immediate.

Note that $q \mapsto \#\{G_h = q\}$ is a Lebesgue measurable function, see e.g.\ \cite{rado}. This follows in essence from the fact that $G_h(B)$ is a Suslin set in case $B$ is a Borel set and Suslin sets are Lebesgue measurable.

\begin{cor}
Let $G : U \to \R^3$ be an embedding of a bounded open set $U \subset \R^2$ such that $G_h$ is of bounded variation. Then $G_h$ is of essentially bounded variation.
\end{cor}

\begin{proof}
For almost every point $q \in \R^2$ it holds that $G_h^{-1}(q)$ is finite. Take any such point and label the preimages by $p_1,\dots,p_n$. Let $r > 0$ such that the balls $\B(p_i,r)$, $i=1,\dots,n$, are pairwise disjoint and contained in $U$. By the locality property for the degree it is obvious that
\[ \ind{G_h}{p_i} \defl \degr{q}{G_h}{D} \]
is independent of the choice of an open neighborhood $D \subset \B(p_i,r)$ of $p_i$. Proposition~\ref{topdisc2} implies that
\[ \ind{G_h}{p_i} \in \{-1,0,1\} \]
by constructing an appropriate domain $D$. In order to construct $D$ take $s$ small enough such that
\[ G_h^{-1}(\B(q,s)) \cap \partial \B(p_i,r) = \emptyset. \]
Define
\[ K \defl G_h^{-1}(\B(q,s)) \cap \B(p_i,r) \]
and let $K'$ be the union of $K$ with all the bounded components of $\R^2 \setminus K$. Then taking the interior of $K'$ for $D$ works fine. If $D$ is any indicator domain for $(q,G_h,U)$, then again by the locality and sum property for the degree
\[ \degr{q}{G_h}{D} = \sum_{p_i \in D} \ind{G_h}{p_i}. \]
For the supremum over all systems $\mathcal S$ of pairwise disjoint indicator domains we obtain
\[ K(q,G_h,U) = \sup_{\mathcal S} \sum_{D \in \mathcal S} |\degr{q}{G_h}{D}| \leq \sum_{i=1}^n |\ind{G_h}{p_i}| \leq \#\{G_h = q\}. \]
This estimate is true for almost all $q \in \R^2$. Hence, $K(q,G_h,U)$ is integrable because $\#\{G_h = q\}$ is.
\end{proof}

\bibliographystyle{bibstyle}
\bibliography{refs}

\end{document}